\documentclass[11pt,reqno]{amsart}
\usepackage{amsmath,amssymb}
\makeatletter \oddsidemargin.9375in \evensidemargin \oddsidemargin
\begin{document}
\begin{center}
\thispagestyle{empty} \setcounter{page}{1} {\large\bf On the
stability of J$^*-$derivations \vskip.20in
{\Small \bf $^1$M. Eshaghi Gordji, $^2$S. Shams, $^3$A. Ebadian and $^4$M. B. Ghaemi } \\[2mm]

{\footnotesize  $^1$Department of Mathematics,
Semnan University,\\ P. O. Box 35195-363, Semnan, Iran\\
e-mail: {\tt  madjid.eshaghi@gmail.com}}

{\footnotesize $^{2,3}$Department of Mathematics,
Urmia University, Urmia, Iran\\
[-1mm] e-mail: {\tt ebadian.a@gmail.com, Sa40shams@yahoo.com}}

{\footnotesize $^4$Department of Mathematics, Iran University of
Science and Technology, Narmak, Tehran, Iran\\
[-1mm] e-mail: {\tt mghaemi@iust.ac.ir}}

 }
\end{center}
\vskip 5mm \noindent{\footnotesize{\bf Abstract.} In this paper, we
establish  the stability and superstability of $J^*-$derivations in
$J^*-$algebras for the generalized  Jensen--type functional equation
$$rf(\frac{x+y}{r})+rf(\frac{x-y}{r})= 2f(x).$$
Finally, we investigate the stability of $J^*-$derivations by using
the fixed point alternative.
 \vskip.10in

\vskip.10in
\newtheorem{df}{Definition}[section]
\newtheorem{rk}[df]{Remark}
\newtheorem{lem}[df]{Lemma}
\newtheorem{thm}[df]{Theorem}
\newtheorem{pro}[df]{Proposition}
\newtheorem{cor}[df]{Corollary}
\newtheorem{ex}[df]{Example}
\setcounter{section}{0} \numberwithin{equation}{section} \vskip
.2in
\begin{center}
\section{Introduction}
\end{center}
Let $\mathcal H$, $\mathcal K$ be two Hilbert spaces and let
$\mathcal B(\mathcal H,\mathcal K)$ be the space of all bounded
operators from $\mathcal H$ into $\mathcal K$. By a $J^*-$algebra we
mean a closed subspace $\mathcal A$ of $\mathcal B(\mathcal
H,\mathcal K)$ such that $xx^*x\in \mathcal A$ whenever $x\in
\mathcal A$. Many familiar spaces are $J^*-$algebras \cite{H1}. Of
course $J^*-$algebras are not algebras in the ordinary sense.
However from the point of view they may be considered a
generalization of $C^*-$algebras; see \cite{H3, H1, H2}. In
particular any Hilbert space may be thought of as a $J^*-$algebra
identified with $\mathcal L(\mathcal H,\Bbb C)$. Also any
$C^*-$algebra in $\mathcal B(\mathcal H)$ is a $J^*-$algebra. Other
important examples of $J^*-$algebras are the so--called Cartan
factors of type I, II, III and IV. A $J^*-$derivation on a
$J^*-$algebra $\mathcal A$  is defined to be a linear mapping
$d:\mathcal A\to \mathcal A$ such that
$$d(aa^*a)=d(a)a^*a+a(d(a))^*a+aa^*d(a)$$
for all $a\in  \mathcal A$.\\ In particular, every $^*-$derivation
on a $C^*-$algebra is a $J^*-$derivation.

The stability of functional equations was first introduced  by S. M.
Ulam \cite{U} in 1940. More precisely, he proposed the following
problem: Given a group $G_1,$ a metric group $(G_2,d)$ and a
positive number $\epsilon$, does there exist a $\delta>0$ such that
if a function $f:G_1\longrightarrow G_2$ satisfies the inequality
$d(f(xy),f(x)f(y))<\delta$ for all $x,y\in G1,$ then there exists a
homomorphism $T:G_1\to G_2$ such that $d(f(x), T(x))<\epsilon$ for
all $x\in G_1?$ As mentioned above, when this problem has a
solution, we say that the homomorphisms from $G_1$ to $G_2$ are
stable. In 1941, D. H. Hyers \cite{H} gave a partial solution of
$Ulam^{,}s$ problem for the case of approximate additive mappings
under the assumption that $G_1$ and $G_2$ are Banach spaces. In
1978, Th. M. Rassias \cite{R1} generalized the theorem of Hyers by
considering the stability problem with unbounded Cauchy differences.
This phenomenon of stability that was introduced by Th. M. Rassias
\cite{R1} is called the Hyers--Ulam--Rassias stability. According to
Th. M. Rassias theorem:
\begin{thm}\label{t1} Let $f:{E}\longrightarrow{E'}$ be a mapping from
 a norm vector space ${E}$
into a Banach space ${E'}$ subject to the inequality
$$\|f(x+y)-f(x)-f(y)\|\leq \epsilon (\|x\|^p+\|y\|^p)$$
for all $x,y\in E,$ where $\epsilon$ and p are constants with
$\epsilon>0$ and $p<1.$ Then there exists a unique additive
mapping $T:{E}\longrightarrow{E'}$ such that
$$\|f(x)-T(x)\|\leq \frac{2\epsilon}{2-2^p}\|x\|^p$$ for all $x\in E.$
If $p<0$ then inequality $(1.3)$ holds for all $x,y\neq 0$, and
$(1.4)$ for $x\neq 0.$ Also, if the function $t\mapsto f(tx)$ from
$\Bbb R$ into $E'$ is continuous for each fixed $x\in E,$ then T
is linear.
\end{thm}
During the last decades several stability problems of functional
equations have been investigated by many mathematicians. A large
list of references concerning the stability of functional
equations can be found in \cite{Cz, H-R, H-I-R, J, R2}.\\

Recently, C$\breve{a}$dariu and Radu  applied the fixed point method
to the investigation of the  functional equations. (see also
\cite{C-R1, C-R2, C-R3, P-R, Ra, Ru}). In \cite{P3}, Park establish
the stability of homomorphisms between
 $C^*-$algebras (see also \cite{M-P, P4, P-R}).
In section 2 of the present paper, we establish  the stability and
superstability of $J^*-$derivations in $J^*-$algebras for the
generalized Jensen--type functional equation
$$rf(\frac{x+y}{r})+rf(\frac{x-y}{r})= 2f(x). \eqno(1.1)$$
 In section 3, we will use the
fixed point alternative of C\u{a}dariu  and Radu to prove the
stability and superstability  of $J^*-$derivations on $J^*-$algebras
for the generalized Jensen--type functional equation (1.1).

Throughout this paper assume that $\mathcal A$ is a $J^*-$algebra.

 \vskip 5mm
\section{Stability}

We start our work by a theorem in superstability of
$J^*-$derivations.

\begin{thm}\label{t1}
Let $r,s\in (1,\infty),$ and let  $D:\mathcal A\to \mathcal A$ be a
mapping for which $D(sa)=sD(a)$ for all $a\in  \mathcal A.$ Suppose
there exists a function $\phi:\mathcal A^3\to [0,\infty)$ such that
$$lim_ns^{-n}\phi(s^na,s^nb,s^nc)=0,$$
$$\|r\mu D(\frac{a+b}{r})+r\mu D(\frac{a-b}{r})-2D(\mu a)+D(cc^*c)-D(c)(c)^*c
-cD(c)^*c-cc^*D(c)\|\leq \phi(a,b,c), \eqno (2.1)$$ for all $\mu \in
\Bbb T$ and all $a,b,c \in  \mathcal A.$ Then $D$ is a
$J^*-$derivation.
\end{thm}
\begin{proof}
Put $\mu=a=b=0$ in (2.1). Then
\begin{align*}\|&D(cc^*c)-D(c)c^*c-cD(c^*)c-cc^*D(c)\|\\&=\frac{1}{s^{3n}}\|D((s^nc)(s^nc^*)(s^nc))-D(s^nc)(s^nc^*)(s^nc)
-(s^nc)D(s^nc^*)(s^nc)\\
&-(s^nc)(s^nc^*)D(s^nc)\|\leq  \frac{1}{s^{3n}}\phi(0,0,s^nc)\leq
\frac{1}{s^n}\phi(0,0,s^nc)
\end{align*}
for all $c \in  \mathcal A.$  The right--hand side tends to zero as
$n\to \infty.$ So
$$D(cc^*c)=D(c)c^*c+cD(c^*)c+cc^*D(c)$$
for all $c\in  \mathcal A.$ Similarly, one can show that
$$r\mu D(\frac{a+b}{r})+r\mu D(\frac{a-b}{r})=2D(\mu a)$$
for all $\mu \in \Bbb T$ and all $a,b \in  \mathcal A.$ Hence, $D$
is linear.
\end{proof}

\begin{thm}\label{t1}
Let $r\in (1,\infty),$ and let  $f:\mathcal A\to \mathcal A$ be a
mapping with $f(0)=0$ for which
 there exists a function
$\phi:\mathcal A^3\to [0,\infty)$ such that
$$\Phi(a,b,c):=\sum_0^{\infty} 2^{-n}\phi(2^na,2^nb,2^nc)<\infty,$$
$$\|r\mu f(\frac{a+b}{r})+r\mu f(\frac{a-b}{r})-2f(\mu a)+f(cc^*c)-f(c)(c)^*c
-cf(c)^*c-cc^*f(c)\|\leq \phi(a,b,c), \eqno (2.2)$$ for all $\mu \in
\Bbb T$ and all $a,b,c \in  \mathcal A.$ Then there exists a unique
$J^*-$derivation $D:\mathcal A\to \mathcal A$ such that
$$\|f(a)-D(a)\| \leq \Phi(a,a,0)\eqno(2.3)$$
for all $a \in  \mathcal A.$
\end{thm}
\begin{proof}
Put $\mu=1$ and $b=c=0$ in (2.2). It follows that
$$\|f(a)-r^{-1}f(ra)\|\leq \frac{1}{2}\phi(ra,0,0)\eqno (2.4)$$
for all $a\in  \mathcal A.$ By induction, we can show that
$$\|f(a)-r^{-n}f(r^na)\|\leq \frac{1}{2}\sum_1^n\phi(r^na,0,0)\eqno (2.5)$$
for all $a\in  \mathcal A.$ Replace $a$ by $a^m$ in (2.5) and then
divide by $r^m$, we get
$$\|f(a^m)-r^{-n-m}f(r^{n+m}a)\|\leq \frac{1}{2r^m}\sum_m^{m+n}\phi(r^{k}a,0,0)$$
for all $a\in  \mathcal A.$ Hence, $\{r^{-n}f(r^na)\}$ is a Cauchy
sequence. Since $A$ is complete, then
$$D(a):=lim_n r^{-n}f(r^na)$$
exists for all $a\in  \mathcal A.$ By using (2.1) one can show that

\begin{align*}\|&rD(\frac{a+b}{r})+ rD(\frac{a-b}{r})- 2D( a)\| \\
&=lim_n
\frac{1}{r^n}\|rf(r^{n-1}(a+b))+rf(r^{n-1}(a-b))-2f(r^na)\|\\
&\leq lim_n\frac{1}{r^n}\phi(r^na,r^nb,0)=0
\end{align*}
for all $a,b \in  \mathcal A.$ So $$rD(\frac{a+b}{r})+
rD(\frac{a-b}{r})= 2D( a)$$ for all $a,b \in  \mathcal A.$ Put
$U=\frac{a+b}{r}, V=\frac{a-b}{r}$ in above equation, we get
$$r(D(U)+D(V))=2rD(\frac{U+V}{2}) $$ for all $U,V\in  \mathcal A.$
 Hence, $D$ is
Cauchy additive. On the other hand, we have
$$\|D(\mu a)-\mu D(a)\|=lim_n
\frac{1}{r^n}\|f(\mu r^n a)-\mu f(r^n a)\|\leq
lim_n\frac{1}{r^n}\phi(r^na,r^na,0)=0$$ for all $\mu \in \Bbb T$,
and all $a\in  \mathcal A.$ So it is easy to show that  $D$ is
linear. It follows from (2.1) that
\begin{align*}\|&D(cc^*c)-D(c)c^*c-cD(c^*)c-cc^*D(c)\|\\
&=lim_n\|\frac{1}{r^{3n}}f((r^nc)(r^nc^*)(r^nc))-\frac{1}{r^{n}}f(r^nc)\frac{r^nc^*}{r^{n}}\frac{r^nc}{r^{n}}
-\frac{r^nc}{r^{n}}\frac{1}{r^{n}}f(r^nc^*)\frac{r^nc}{r^{n}}\\
&-\frac{r^nc}{r^{n}}\frac{r^nc^*}{r^{n}}\frac{1}{r^{n}}f(r^nc)\|
\leq lim_n \frac{1}{r^{3n}}\phi(0,0,r^nc)\leq lim_n \frac{1}{r^n}\phi(0,0,r^nc)\\
&=0
\end{align*}
for all $c \in  \mathcal A.$  Thus $D:\mathcal A\to \mathcal A$ is a
$J^*-$derivation satisfying (2.3), as desired.
\end{proof}

We prove the following Hyers--Ulam--Rassias stability problem for
$J^*-$derivations on  $J^*-$algebras.

\begin{cor}\label{t2}
Let $p\in (0,1), \theta \in [0,\infty)$ be real numbers. Suppose
$f:A \to A$ satisfies $$ \|r\mu f(\frac{a+b}{r})+r\mu
f(\frac{a-b}{r})-2f(\mu a)+f(cc^*c)-f(c)(c)^*c -cf(c)^*c-cc^*f(c)\|
\leq \theta(\|a\|^p+\|b\|^p+\|c^p\|),$$ for all $\mu \in \Bbb T$ and
all $a,b,c \in  \mathcal A.$  Then there exists a unique
$J^*-$derivation $D:\mathcal A\to \mathcal A$ such that
$$\|f(a)-D(a)\| \leq \frac{2^p\theta}{2^{p-1}-1}\|a\|^p$$
for all $a \in  \mathcal A.$
\end{cor}
\begin{proof}
Setting $\phi(a,b,c):=\theta(\|a\|^p+\|b\|^p+\|c\|^p)$ for all
$a,b,c \in  \mathcal A,$ in above theorem.
\end{proof}

 \vskip 5mm
\section{Stability by Using Alternative Fixed Point}

 Before proceeding to the
main results of this section, we will state the following theorem.

\begin{thm}\label{t2}(The alternative of fixed point \cite{C-R}).
Suppose that we are given a complete generalized metric space
$(\Omega,d)$ and a strictly contractive mapping
$T:\Omega\rightarrow\Omega$ with Lipschitz constant $L$. Then for
each given $x\in\Omega$, either

$d(T^m x, T^{m+1} x)=\infty~$ for all $m\geq0,$\\
 or other exists a natural number $m_{0}$ such that\\
$\star\hspace{.25cm} d(T^m x, T^{m+1} x)<\infty ~$for all $m \geq m_{0};$ \\
 $\star\hspace{.1cm}$ the sequence $\{T^m x\}$ is convergent to a fixed point $y^*$ of $~T$;\\
$\star\hspace{.2cm} y^*$is the unique fixed point of $~T$ in the
set $~\Lambda=\{y\in\Omega:d(T^{m_{0}} x, y)<\infty\};$\\
$\star\hspace{.2cm} d(y,y^*)\leq\frac{1}{1-L}d(y, Ty)$ for all
$~y\in\Lambda.$
\end{thm}

\begin{thm}\label{t1}
Let $f:\mathcal A\to \mathcal A$ be a mapping for which there exists
a function $\phi:\mathcal A^3\to [0,\infty)$ such that

$$\|r\mu f(\frac{a+b}{r})+r\mu f(\frac{a-b}{r})-2f(\mu a)+f(cc^*c)-f(c)(c)^*c
-cf(c)^*c-cc^*f(c)\|\leq \phi(a,b,c), \eqno (3.1)$$ for all $\mu \in
\Bbb T$ and all $a,b,c \in  \mathcal A.$ If  there exists an $L<1$
such that
$$\phi(a,b,c)\leq rL \phi(\frac{a}{r},\frac{b}{r},\frac{c}{r}) \eqno
(3.2)$$ for all $a,b,c\in  \mathcal A,$ then there exists a unique
$J^*-$derivation $D:\mathcal A\to \mathcal A$ such that
$$\|f(a)-D(a)\| \leq \frac{L}{1-L}\phi(a,0,0)\eqno(3.3)$$
for all $a \in  \mathcal A.$
\end{thm}
\begin{proof}
Put $\mu =1, b=c=0$ in (3.1) to obtain
$$\|2rf(\frac{a}{r})-2f(a)\|\leq \phi(a,0,0)\eqno (3.4)$$
for all $a\in  \mathcal A.$ Hence,
$$\|\frac{1}{r}f(ra)-f(a)\|\leq \frac{1}{2r} \phi(ra,0,0)\leq L\phi(ra,0,0)\eqno (3.5)$$
for all $a\in  \mathcal A.$\\
Consider the set $X:=\{g\mid g:A\to \mathcal A\}$ and introduce the
generalized metric on X:
$$d(h,g):=inf\{C\in \Bbb R^+:\|g(a)-h(a)\|\leq C\phi(a,0,0) \forall a\in  \mathcal A\}.$$
It is easy to show that $(X,d)$ is complete. Now we define  the
linear mapping $J:X\to X$ by $$J(h)(a)=\frac{1}{r}h(ra)$$ for all
$a\in  \mathcal A$. By Theorem 3.1 of \cite{C-R}, $$d(J(g),J(h))\leq
Ld(g,h)$$ for
all $g,h\in X.$\\
It follows from (3.5) that  $$d(f,J(f))\leq L.$$ By Theorem 3.1, $J$
has a unique fixed point in the set $X_1:=\{h\in X: d(f,h)<
\infty\}$. Let $D$ be the fixed point of $J$. $D$ is the unique
mapping with
$$D(ra)=rD(a)$$ for all $a\in  \mathcal A$ satisfying there exists $C\in
(0,\infty)$ such that
$$\|D(a)-f(a)\|\leq C\phi(a,0,0)$$ for all $a\in  \mathcal A$. On the other hand we
have $lim_n d(J^n(f),D)=0$. It follows that
$$lim_n\frac{1}{2^n}{f(2^na)}=D(a)\eqno (3.6)$$
for all $a\in  \mathcal A$. It follows from $d(f,h)\leq
\frac{1}{1-L}d(f,J(f)),$ that $$d(f,h)\leq \frac{L}{1-L}.$$ This
implies the inequality (3.3).

It follows from (3.2) that
$$lim_j r^{-j}\phi(r^ja,r^jb,r^jc)=0 \eqno(3.7)$$ for all $a,b,c\in  \mathcal A.$\\

By a same reasoning as Theorem 2.2, one can show  that the mapping
$D:\mathcal A\to \mathcal A$ is  a $J^*-$derivation satisfying
(3.3), as desired.
\end{proof}
We prove the following Hyers--Ulam--Rassias stability problem for
$J^*-$derivations on  $J^*-$algebras.

\begin{cor}\label{t2}
Let $p\in (0,1), \theta \in [0,\infty)$ be real numbers. Suppose
$f:A \to A$ satisfies $$ \|r\mu f(\frac{a+b}{r})+r\mu
f(\frac{a-b}{r})-2f(\mu a)+f(cc^*c)-f(c)(c)^*c -cf(c)^*c-cc^*f(c)\|
\leq \theta(\|a\|^p+\|b\|^p+\|c^p\|),$$ for all $\mu \in \Bbb T$ and
all $a,b,c \in  \mathcal A.$  Then there exists a unique
$J^*-$derivation $D:\mathcal A\to \mathcal A$ such that
$$\|f(a)-D(a)\| \leq \frac{2^p\theta}{2-2^p}\|a\|^p$$
for all $a \in  \mathcal A.$
\end{cor}
\begin{proof}
Setting $\phi(a,b,c):=\theta(\|a\|^p+\|b\|^p+\|c^p)$ all $a,b,c \in
\mathcal A.$ Then by $L=2^{p-1}$, we get the desired result.
\end{proof}

Now we establish the superstability of $J^*-$derivations by using
the alternative of fixed point.

\begin{thm}\label{t3}
Let $s>1$, and let  $f:\mathcal A\to \mathcal A$ be a mapping
satisfying $f(sx)=sf(x)$ for all $x\in  \mathcal A$. Let
$\phi:\mathcal A^3\to [0,\infty)$  be a mapping satisfying  (3.1).
If  there exists an $L<1$ such that
$$\phi(x,y,z)\leq rL \phi(\frac{x}{r},\frac{y}{r},\frac{z}{r}) $$
for all $x,y,z\in  \mathcal A,$ then $f$ is  a  $J^*-$derivation.
\end{thm}
\begin{proof} It is similar to the proof of Theorem 2.1.
\end{proof}

\begin{cor}\label{t4}
Let  $r,p\in (0,1), \theta \in [0,\infty)$ be real numbers. Suppose
$f:A \to A$  is a mapping satisfying $f(rx)=rf(x)$ for all $x\in
\mathcal A$. Let  $\phi:\mathcal A^3\to [0,\infty)$  be a mapping
satisfying (3.1). Then $f$ is  a $J^*-$derivation.

\end{cor}
\begin{proof}
Setting $\phi(x,y,z):=\theta(\|x\|^p+\|y\|^p+\|z\|^p)$ all $x,y,z
\in  \mathcal A.$ Then by $L=2^{p-1}$ in above theorem, we get the
desired result.
\end{proof}

{\small

}

\begin{thebibliography}{99}


\bibitem{C-R} L. C\u adariu and V. Radu, \textit{On the stability of the Cauchy functional equation:
a fixed point approach,} { Grazer Mathematische Berichte}
\textbf{346} (2004),  43--52.



\bibitem{C-R1} L. C$\breve{a}$dariu, V. Radu, \textit{The fixed points method for the stability of some functional
equations,} Carpathian Journal of Mathematics \textbf{23} (2007),
 63--72.




\bibitem{C-R2} L. C$\breve{a}$dariu, V. Radu, \textit{Fixed points and the stability of quadratic
functional equations,} Analele Universitatii de Vest din Timisoara
\textbf{41} (2003),  25--48.


\bibitem{C-R3} L. C$\breve{a}$dariu and V. Radu, \textit{Fixed points and the stability of Jensen's functional
equation,} { J. Inequal. Pure Appl. Math.} \textbf{ 4 } (2003), Art.
ID 4.




\bibitem{Cz} S. Czerwik, \textit{Functional Equations and Inequalities in Several Variables},
World Scientific, River Edge, NJ, 2002.

\bibitem {H3} M. Elin, L. Harris, S. Reich and D. Shoikhet, \textit{ Evolution Equations
and Geometric Function Theory in J*-algebras,} J. Nonlinear Convex
Anal., \textbf{3} (2002), 81--121.

\bibitem {H1} L.A. Harris, \textit{Bounded symmetric homogeneous domains in
infinite-dimensional spaces}, Lecture Notes in Mathematics, vol.
364, Springer, Berlin, 1974.

\bibitem {H2} L.A. Harris, \textit{Operator Siegel domains,} Proc.
Roy. Soc. Edinburgh Sect. A. \textbf{79} (1977) 137–156.




\bibitem{H} D.H. Hyers, \textit{On the stability of the linear functional equation},
Proc. Nat. Acad. Sci. U.S.A. \textbf{27} (1941), 222--224.

\bibitem{H-R} D.H. Hyers and Th.M. Rassias, \textit{Approximate
homomorphisms}, Aequationes Math. \textbf{44} (1992), no. 2-3,
125--153.

\bibitem{H-I-R} D.H. Hyers, G. Isac and Th.M. Rassias, \textit{Stability of Functional Equations in Several Variables},
Birkh\"auser, Basel, 1998.






\bibitem{J} S.-M. Jung, \textit{Hyers--Ulam--Rassias Stability of Functional
Equations in Mathematical Analysis}, Hadronic Press, Palm Harbor,
2001.

\bibitem{M-P} M. S. Moslehian and C. Park, \textit{On the stability of
$J^*-$homomorphisms,} Nonlinear Analysis, \textbf{63} (2005),
42-–48.



\bibitem{P3} C. Park, \textit{Isomorphisms between C*-ternary algebras,} { J. Math. Anal.
Appl.} \textbf{327} (2007), 101-115.


\bibitem{P4} C. Park, \textit{Homomorphisms between Poisson JC*-algebras}, { Bull. Braz. Math. Soc.} \textbf{36} (2005) 79-97.
MR2132832 (2005m:39047)


\bibitem{P-R} C. Park and J. M. Rassias, \textit{Stability of the Jensen-type functional equation in $C^*-$algebras:
a fixed point approach,}  Abstract and Applied Analysis Volume 2009,
Article ID 360432, 17 pages.

\bibitem {Ra} V. Radu, \textit{The fixed point alternative and the stability of functional
equations,}
{ Fixed Point Theory} \textbf{4} (2003),  91--96.


\bibitem{R1} Th.M. Rassias, \textit{On the stability of the linear mapping in Banach
spaces,}
 { Proc. Amer. Math. Soc.} \textbf{72} (1978) 297--300.


\bibitem{R2} Th.M. Rassias, \textit{On the stability of functional equations
and a problem of Ulam}, Acta Appl. Math. \textbf{62} (2000), no. 1,
23--130.


\bibitem {Ru} I.A. Rus, \textit{ Principles and Applications of Fixed Point Theory,} Ed.
Dacia, Cluj-Napoca, 1979 (in Romanian).


\bibitem{U} S. M. Ulam, \textit{Problems in Modern Mathematics,}{
Chapter VI, science ed. Wiley, New York,} 1940.



\end{thebibliography}
\end{document}